\documentclass[11pt]{amsart}
\usepackage{amssymb,amsmath, graphicx, verbatim,epsfig,color}
\usepackage{hyperref}
\hypersetup{
    colorlinks,%
    citecolor=black,%
    filecolor=black,%
    linkcolor=black,%
    urlcolor=blue
}
\usepackage[latin1]{inputenc}
 \usepackage[all,cmtip]{xy}
\usepackage{todonotes}

\newtheorem{theorem}{Theorem}[section]
\newtheorem{pr}[theorem]{Proposition}
\newtheorem{cor}[theorem]{Corollary}

\newtheorem{lm}[theorem]{Lemma}

\theoremstyle{definition}
\newtheorem{define}[theorem]{Definition}
\theoremstyle{remark}

\newcommand{\barr}{\overline}

\title[Petri map for vector bundles near good bundles]
{Petri map for vector bundles near good bundles}

\author[A. Castorena]{Abel Castorena}
\address{Centro de Ciencias Matem\'aticas( Universidad Nacional Auton\'oma de M\'exico, Campus Morelia); Apartado Postal 61-3 (Xangari), 58089 Morelia, Michoac\'an}
\email{abel@matmor.unam.mx}

\author[A. L\'opez]{Alberto L\'opez Mart\'in}
\author[M. Teixidor]{Montserrat Teixidor i Bigas} \address{Department of Mathematics, Tufts University, Bromfield-Pearson Hall, 503 Boston Avenue, Medford, MA 02155}
\email{alberto.lopez@tufts.edu}
\email{montserrat.teixidorbigas@tufts.edu}

\subjclass[2010]{Primary 14H51 $\cdot$ Secondary 14H60}

\thanks{The first author was supported by grants IN100211-2 (PAPIIT-UNAM) and 166158 (CONACyT)}
\date{\today}

\begin{document}
\begin{abstract} Let $C$ be a projective and nonsingular curve of genus $g$. Denote by $\omega$ the canonical line bundle on $C$. Consider the locus $\mathcal B^k_{r,d}$ of stable vector bundles of rank $r$ and degree $d$ with at least $k$ independent sections on $C$. In this paper we show that when $C$ is generic and under some conditions on the degree and genus, there exists a component $B$ of $\mathcal B^k_{r,d}$ of the expected dimension, such that for a generic vector bundle $E$ of $B$, the Petri map $H^0(C,E)\otimes H^0(C,\omega \otimes E^*)\rightarrow H^0(C,\omega \otimes E\otimes E^*)$ is injective.
 
 \end{abstract}
\maketitle
\section{Introduction}

\bibliographystyle{alpha}
Let $C$ be a projective nonsingular genus $g$ curve defined over an algebraically closed field. Let $\mbox{Pic}^d(C)$ be the Picard variety that parametrizes isomorphism classes of degree $d$ line bundles on $C$. Let $W_d^r$ be the subvariety of $\mbox{Pic}^d(C)$ of all line bundles of degree $d$ with at least $r+1$ sections. This set, known as the Brill-Noether locus, has a natural scheme structure of determinantal variety. For $L\in\mbox{Pic}^d(C)$, the Petri map is defined as the natural cup-product \begin{align}H^0(C,L)\otimes H^0(C,\omega\otimes L^{-1})\longrightarrow H^0(C,\omega),\end{align}where $\omega$ denotes the canonical line bundle on $C$. The vector space $H^0(C,\omega)$ can be identified with the dual of  the tangent space to $\mbox{Pic}^d(C)$. Suppose that $h^0(C,L)\geq r+1$. The tangent space to $W^r_d$ at the point $L$ is the orthogonal to the image of the Petri map. The injectivity of the Petri map on a generic curve for every line bundle was conjectured by Petri in \cite{Pet25} and first proved by Gieseker\ \cite{Gie82}.  The Gieseker-Petri Theorem states that $W^r_d$ is empty when its expected dimension, given by the Brill-Noether number $\rho(g,d,r)= g-(r+1)(g-d+r)$, is negative and of the expected dimension when nonempty. 

The generalization of classical Brill-Noether theory to higher rank considers the loci of rank $r$ degree $d$ stable vector bundles $E$ that have at least a given number $k$ of sections. These loci have also a natural scheme structure as locally defined determinantal varieties and therefore have an expected dimension and singular locus. 
Most of the properties of classical Brill-Noether theory no longer hold, that is, the Brill-Noehter
 loci on the generic  curve may be empty  when the expected dimension is positive and nonempty when it is negative, 
 they may be reducible and 
the singular locus may be larger than expected.

 The equivalent of the Petri map in the case of higher rank bundles is defined by the natural cup-product
\begin{align}P_E: H^0(C,E)\otimes H^0(C,\omega \otimes E^*)\longrightarrow H^0(C,
\omega \otimes E\otimes E^*).
\end{align}
The orthogonal to the image of this map is the tangent space at $E$ to the locus  $\mathcal B^k_{r,d}$ 
of vector bundles of rank $r$ degree $d$ with 
$k=h^0(E)$ sections. If the Petri map is injective, then the Brill-Noether locus for $k$ is nonsingular and of the 
expected dimension at $E$. It is not true, however, that  the Petri map is injective for every stable vector bundle
 on a generic curve. In fact, for certain values of the rank, degree and genus, one can find vector bundles
  on each curve for which the map is not injective (see \cite{Tei91a}).

Despite all the abnormalities cited above, it was proved in \cite{Tei05,Tei08b} that for a large range 
of values of $r, k, d$ and $g$, and for a generic curve, there exist components of the expected dimension 
in the Brill-Noether locus  $\mathcal B^k_{r,d}$ of stable vector bundles. 
Existence of a component of the right dimension has also been proved by other methods under additional conditions (see, for example, \cite[$\S$4--6]{Iwona} for an overview).

 For the case of rank two, Ciliberto and Flamini \cite{CilFlam} show again  the existence of these components
   using the tools of ruled surfaces. Moreover, they show that for the generic point of the components 
   they construct, the Petri map is injective. For vector bundles with number of sections at most equal to the rank or of slope at most two, the components are also
    shown to be generically nonsingular in \cite{Coh, Coh2}.
   
    The goal of this paper is to show that injectivity 
   of the Petri map for the generic vector bundle of one component holds 
   in much greater generality. In particular, we will see that it holds  in the range for which
     existence was proved in \cite{Tei05,Tei08b}.

We have the following

\begin{theorem} Let $C$ be a generic nonsingular curve of genus $g\ge 2$.
 Let $d,r,k$ be positive integers with  $ k>r$. Write
$$d=rd_1+d_2,\ k=rk_1+k_2, \ d_2<r,\ k_2<r$$ and all $d_i,\ k_i$
nonnegative integers. Assume that

\vspace{3mm}
\begin{center}
\begin{tabular}{r l}
\textnormal{(*)}& $g-(k_1+1)(g-d_1+k_1-1)\ge 1 ,\ 0\not= d_2\ge k_2;$\\
\textnormal{(**)}& $g-k_1(g-d_1+k_1-1)> 1 ,\ d_2= k_2=0;$\\
\textnormal{(***)}& $g-(k_1+1)(g-d_1+k_1)\ge 1 ,\ d_2< k_2.$
\end{tabular}
\end{center}
\vspace{2mm}
Then, the moduli space of
coherent systems of rank $r$, degree $d$, and with $k$ sections on
$C$, is nonempty and has one component of the expected dimension
which is smooth at the generic point, that is, a component for which 
 the Petri map $P_E$ is injective at the generic point. \end{theorem} 
 
 For rank two, if both the number of sections of the vector bundle and the number of sections of 
 its Serre dual are at least two, one can get a slightly better result (see \cite{Tei08b}).
 
 \begin{theorem} Let $C$ be a generic nonsingular curve of genus $g\ge 2$.
 Let $d,k$ be positive integers with  $ k\ge 2$ and $k-d+2(g-1)\ge 2$. 
 Assume that the expected dimension of the locus of vector bundles with these sections is 
at least 1 for $d$ odd and at least 5 for $d$ even. 
  Then, the locus of rank two vector bundles of degree $d$ with $k$ sections
   on $C$ is nonempty and has one component of the expected dimension
which is smooth at the generic point, that is, a component for which 
 the Petri map $P_E$ is injective at the generic point. \end{theorem}

 The two theorems above  will be a  consequence of the more general result \ref{Petri}
  and the constructions in \cite{Tei05, Tei08b}. In these two papers, the construction of a component of
  $\mathcal B^k_{r,d}$ is obtained by  degenerating the generic curve $C$ to a reducible curve $C_0$ 
   and constructing a limit stable vector bundle with limit sections
  on the reducible curve.
   It turns out that this limit linear series satisfies good properties, 
   in particular its restriction to every component is semistable of a fairly general type 
   so that the result below is applicable.

\begin{pr}\label{Petri} Assume that we have a family of curves for which the generic fiber is 
 a generic curve $C$  of genus $g$ and the special fiber $C_0$ is a chain of elliptic curves.
  Consider the locus $\mathcal B^k_{r,d}$ consisting of stable limit linear series  of rank $r$ and degree $d$
   with at least $k$ independent sections on the family. Suppose that there exists a component $B$ of
    $\mathcal B^k_{r,d}$ 
 such that for the generic point $E\in B$, the restriction of the vector bundle to each elliptic component of
 $C_0$  is a direct sum of generic indecomposable vector bundles all of the same slope or a direct sum of 
 line bundles that are generics or all are equal, then the Petri map 
\begin{align}P_E: H^0(C,E)\otimes H^0(C,\omega \otimes E^*)\rightarrow H^0(C,\omega \otimes E\otimes E^*)\end{align}
 is injective for the generic vector bundle in $B$ over $C$.\end{pr}
 
 The proof of this result uses the theory of limit linear series for vector bundles that we review 
 in the next Section. Then in Section 3, we do an analysis of the  vanishing  at the nodes of a section
  in the kernel, similar to what was done in \cite{EH83, TeiPet} to finish the proof of Prop. \ref{Petri}.
   The arguments, while parallel to those in \cite{TeiPet} are somewhat simpler. In our case, we do not need to 
   investigate all possible cases of limit vector bundles, it suffices to consider those used in 
   the constructions in \cite{Tei05, Tei08b}.
   
   The methods of this paper could potentially be used to bound the dimension of the singular locus of the
    Brill-Noether loci and therefore to bound the dimension of the Brill-Noether loci themselves.

\section{Limit Linear Series}\label{limlinser}

Assume that we have a family of curves ${\mathcal C}\rightarrow S$
where $S$ is the spectrum of a discrete valuation ring. Denote by
$\eta$ the generic point in $S$, by $t$ a generator of the maximal
ideal of the ring and by $\nu$ the discrete valuation.

We assume that ${\mathcal C}_{\eta}$ is nonsingular while the central
fiber is a generic chain of elliptic and rational curves. By this we mean that the 
central fiber is a reducible curve with $M$ irreducible components  
 $C_1,\dots,C_M$ that are either rational or elliptic and so that the curve $C_i$
intersects only $C_{i-1}, C_{i+1}$ at generic points $P_i, Q_i\in C_i$.

 We want to consider linear series for vector bundles of
rank $r$ on our family. On the central fiber, one obtains then a limit linear
series in the sense of \cite{Tei91a, Clay} that generalizes the 
concept of limit linear series of Eisenbud and Harris. As in the case of rank one,
one needs to give not just one space of sections but several. However, unlike in the case
of rank one, the way the restrictions of the vector bundle to each component
glue at the nodes is relevant and so is the way the various sections glue at the nodes.

When one has a line bundle on the family, tensoring with line bundles supported 
on the reducible fiber alters the degrees on each component in an arbitrary way. 
This is no longer true for vector bundles, therefore the degrees of the vector bundles
 on the different components will not be assumed to be all the same.

\begin{define}\label{lls}A limit linear series is defined by the following pieces of data:
\begin{enumerate}
\item[I)]Vector bundles $E_i$ of rank $r$ and degree $d_i$ on each component $C_i$ and $k$-dimensional spaces  $V_i$ of sections of
$E_i$ so that  $\sum _{i=1}^M d_i-r(M-1)a=d$ for a fixed integer $a$,  and the sections of the vector bundles $E_i(-aP_i), E_i(-aQ_i)$ are completely determined by their value at the nodes.

\item[II)]For every node obtained by gluing $Q_i$ and $P_{i+1}$, an isomorphism of the projectivization of the fibers $(E_i)_{Q_i}$ and $(E_{i+1})_{P_{i+1}}$.

\item[III)]Bases of the spaces of sections depending on both the component and the node on the component $s^t_{Q_i}, s^t_{P_{i+1}}$, $t=1,\dots,k$ of the vector spaces $V_i$ and
$V_{i+1}$ so that corresponding sections glue by the isomorphisms above
and $ord_{P_{i+1}}s_{i+1}^t+ord_{Q_{i}}s_{i}^t\ge
a$.
\end{enumerate}
\end{define}

For the reader familiar with limit linear series for line bundles: in that case all $d_i$ are the same and equal to the initial degree $d$, and $a=d$. As the fibers are then one-dimensional, the gluing at the nodes is irrelevant.

The orders of vanishing at a point $P_i$ (resp. $Q_i$) of a linear series will be denoted by $(a_j(P_i))$ (resp. $( a_j(Q_i)))$, $j=1, \dots,k$.  Note that each $a_j$ will appear at most $r$ times. 
If a vanishing appears with order $j$, it means that there is a $j$-dimensional space of sections with this vanishing at $P_i$ (resp. $Q_i$).

\section{Proof of the main result}

We will be using the notation of the previous section. Let us assume that ${\mathcal C}\rightarrow S$ is a family of curves
parametrized by  the spectrum of a discrete valuation ring, where
 $t$ denotes the parameter
and $\nu$ the discrete valuation.
We assume that the central
fiber is a chain of curves as described above.
The objective is to prove the result for the geometric generic fiber, that is, for $\mbox{Spec } \barr K(\eta)$, 
where $K(\eta)$ denotes the field of fractions of the discrete valuation ring and 
$\barr K(\eta)$ is its algebraic closure.

Choose a component $C_i$ of the central fiber and denote by ${\mathcal E}_i$
 the vector bundle on $\pi :{\mathcal
C}\rightarrow S$ whose restriction to all components of the central
fiber, except $C_i$, has trivial sections. With the notation of the previous section, this is 
 the vector bundle whose restriction to $C_i$ is $E_i$.
 
 The proof of \ref{Petri} will follow from an analysis of the vanishing at the nodes 
 of a potential section in the kernel of the Petri map. This follows the pattern of 
 the proof in \cite{EH83, TeiPet}.

\begin{lm} For every component $C_i$, there is a basis $\sigma_j,
j=1,\dots,k$ of $\pi_*{\mathcal E}_i$ such that

a) $ord_{P_i}(\sigma_j)=a_j(P_i)$;

b) for suitable integers $\alpha_j$, $t^{\alpha_j}\sigma _j$ are a
basis of $\pi _*({\mathcal E}_{i+1})$.
\end{lm}

For a proof of this fact, see \cite[Lemma 1.2]{EH83}.

\begin{pr}\label{vansigma}
Let $\sigma_j$ be a basis of $\pi_*{\mathcal
E}_{C_i}$such that $t^{\alpha _j}\sigma _j$ is a basis of
$\pi_*({\mathcal E}_{i+1})$.  Assume that  $E_i$ is a direct sum of  indecomposable
vector bundles of rank $r'$ and degree $r'd_1+d'_2, \ 0\le d'_2<r'$.

Then, the orders of vanishing of the
$\sigma_j$ at the nodes satisfy
\begin{align}\label{ineq}ord_{P_i}(\sigma_j)\le d_1-ord_{Q_i}\sigma_j\le \alpha _j\le
ord_{P_{i+1}}t^{\alpha_j}\sigma_j.\end{align}
Moreover, if equality holds, then $\sigma_j$ vanishes at
$P_i,Q_i$ as a section of $E_i$ to orders adding up to $d_1$.
\end{pr}

For a proof of this proposition, see \cite[Prop. 4.2]{TeiPet}.

It will be relevant to know when the inequalities above are equalities.
 From the last statement in the proposition above,
 this will happen only if there is a section of the vector bundle that 
 vanishes with a certain order, say $a$ at $P$ and order $d_1-a$ at $Q$. This implies 
 that $H^0(E(-aP-(d_1-a)Q)\ge 1$. The number of independent sections satisfying the condition is at 
 most $h^0(E(-aP-(d_1-a)Q)$.
 
 If $E$ is indecomposable of rank $r'$ and 
degree $r'd_1+d'_2, 0\le d'_2<r'$, there is a finite number of
sections (up to a constant) such that
$ord_{P_i}(s)+ord_{Q_i}(s)=d_1$. For
each possible vanishing $a$ at $P$ 
 there are at most $h^0(E(-aP-(d_1-a)Q))=d'_2$ independent sections satisfying 
 the condition. 
 
 If $L$ is a line bundle of degree $d_1$ on an elliptic curve with a section vanishing at $P$ to order
  $a$ and at $Q$ with order $d_1-a$, then $ L= {\mathcal
 O}(aP+(d_1-a)Q)$. If $L$ had another section vanishing at $P$ to order
  $b$ and at $Q$ with order $d_1-b$, then the two divisors 
  $aP+(d_1-a)Q$, $bP+(d_1-b)Q$ would be linearly equivalent and 
  $P-Q$ would be a torsion point in the elliptic curve. In particular, the points
  $P, Q$ would not be generic. So, under the assumption of genericity for the points, 
  there is at most one section of a line bundle whose order of vanishing
  at the two points adds up to the degree. Moreover, when this happens, the line bundle is special.
  Hence, if $E$ is the direct sum of $r$ line bundles of degree $d_1$, 
  the number of independent sections with maximum vanishing at the
nodes adding up to $d_1$ is at most the number of line bundles that are special
 among those  in the decomposition of the vector bundle as direct sum.
 
  The canonical line bundle on an elliptic curve is trivial and on a rational curve there is only
  one line bundle up to isomorphism. It follows that the canonical linear
series has restriction to the  component $C_i$ (\cite{EH83,Wel85}).
$$\omega_{i|C_i}= {\mathcal O}(2(\sum_{k\le i}g(C_k)-1)P_i+2(g-\sum_{k\le
i}g(C_k))Q_i)$$
Therefore, the restriction of a vector bundle to a component is special in the sense above if and only if its Serre dual is special.

We proved the following

\begin{cor}  \label{Remark}  If $E$ is  an indecomposable vector bundle of rank $r'$ and 
degree $r'd_1+d'_2, 0\le d'_2<r'$, for each possible vanishing $a$ at $P$ 
 there are at most $d'_2$ independent sections satisfying equality in (\ref{ineq}).  If $E$ is the direct sum of $r$ line bundles of degree $d_1$, 
  the number of independent sections that satisfy an equality in 
  (\ref{ineq}) is at most the number of line bundles that are special
 among those  in the decomposition of the vector bundle as direct sum. The number of special line bundles appearing 
 in the direct sum decomposition of a vector bundle is the same as the number appearing in the decomposition 
 of its Serre dual.
\end{cor}

Define the order of vanishing of a section $\rho$ of
$\pi_*({\mathcal E}_i)\otimes \pi_*(\omega\otimes {\mathcal E}_i^*)$ as  in  \cite{EH83, TeiPet}:

\begin{define} Let $\{\sigma^j\}$ be a basis of
$\pi_*{\mathcal E}_i$ such that their orders of vanishing are the orders
of vanishing of the linear series  at $P_i$.  Let $\{\tau^l\}$ be a basis of
$\pi_*(\omega\otimes {\mathcal E}_i^*)$ such that their orders of vanishing are the orders
of vanishing of the linear series  at $P_i$.
Write $\rho =\sum_{j,
l}f_{j,l}\sigma^j\otimes \tau^l$ with
$f_{jl}$ functions on the discrete valuation ring. We say $ord_{P_i}\rho \ge \lambda$ if for every
$j,l$ with $f_{j,l}(P_i)\not=0$,
$ord_{P_i}\sigma^j+ord_{P_i}\tau^l\ge \lambda$
\end{define}

 As in
\cite[p. 278]{EH83}, one can identify $\pi_*{\mathcal E_i}, \pi_*(\omega\otimes {\mathcal E}_i^*)$ and $\pi_*
({\mathcal E}_i\otimes \omega\otimes {\mathcal E}_i^*)$ with submodules of $(\pi_*{\mathcal E}_i)_{\eta},  \pi_*(\omega\otimes {\mathcal E}_i^*)_{\eta}$ and
$\pi_* ({\mathcal E}_i\otimes \omega\otimes {\mathcal E}_i^*)_{\eta}$ and also with submodules of
$\pi_*{\mathcal E}_{i+1},  \pi_*(\omega\otimes {\mathcal E}_{i+1}^*)$ and $\pi_*({\mathcal E}_{i+1}\otimes\omega\otimes  {\mathcal E}_{i+1}^*)$. 

Let
$$\rho \in \pi_*({\mathcal E_1}\otimes \omega\otimes {\mathcal E_1}^*) .$$
One can find an integer $b_i$ such that 
$$\rho_i =t^{b_i}\rho \in \pi_*({\mathcal E_i}\otimes \omega\otimes {\mathcal E}_i^*)-t\pi_*({\mathcal E_i}\otimes \omega\otimes {\mathcal E}_i^*).$$

Then for  $a_i=b_{i+1}-b_i$, one has
$$\rho_{i+1} =t^{a_i}\rho_i \in \pi_*({\mathcal E}_{i+1}\otimes \omega\otimes {\mathcal E}_{i+1}^*)-t\pi_*({\mathcal E}_{i+1}\otimes \omega\otimes {\mathcal E}^*)_{i+1}.$$

\begin{pr}
 \label{ordrho} Fix a component $C_i$. Assume
that $\rho=\sum
 f_{jl}\sigma^j\otimes\tau^l$ where
 the $\sigma^j$ are a basis of $\pi_* {\mathcal E}_i$ such that
 $t^{\alpha_j}\sigma^j$ is a basis of $\pi_*{\mathcal E}_{i+1}$ and the 
 $\tau^l$ are a basis of $\pi_* (\omega\otimes {\mathcal E}^*)_i$ such that
 $t^{\beta_l}\tau^l$ is a basis of $\pi_*(\omega\otimes {\mathcal E}_{i+1}^*)$
$$t^{\gamma^i}\rho \in \pi_*({\mathcal E}_i\otimes \omega\otimes {\mathcal E}_i^*)-t (\pi_*({\mathcal E}_i\otimes \omega\otimes {\mathcal E}_i^*)).$$ Then $$ord_{P_i}(\rho )=min_{\{j,l|
\nu(f_{jl})=0\}}ord_{P_i}(\sigma^j)+ord_{P_i}(\tau^l)$$
$$\gamma ^i=max_{\{ j,l\} }(\alpha _j+\beta_l-\nu(f_{jl}).$$
\end{pr}

We now assume that the kernel of the Petri map is non-trivial on the
generic curve. We can then find a section $\rho _{\eta}$ in the
kernel of the Petri map over the generic point. As above, we can
find a $\rho_i$ for each $i, 1\le i\le M$, in the kernel of the map
$$\pi_*({\mathcal E}_i)\otimes \pi_* (\omega\otimes {\mathcal E}_i^*) \rightarrow (\pi_*({\mathcal E}_i\otimes  (\omega\otimes {\mathcal E}_i^*) )),$$
with $\rho_i \notin t\pi_*({\mathcal E}_i\otimes  (\omega\otimes {\mathcal E}_i^*) )$.

\begin{pr}\label{E1}
Assume that the restriction of the vector bundle to the elliptic curve $C_i$
is a direct sum of $h$ generic indecomposable vector bundles of rank $r'$ and degree $r'd_1+d'_2$ where $0\le d'_2<r',\ h=gcd(r,d)$.

Then, $$ord_{P_{i+1}}(\rho_{i+1})\ge ord_{P_{i}}(\rho_{i})+2.$$

\end{pr}

\begin{proof} An indecomposable vector bundle of rank $r'$ and degree $r'd_1+d'_2$ can be deformed
 to a direct sum of $d'_2$ generic line bundles of degree $d_1+1$ and $r'-d'_2$ generic line bundles of
  degree $d_1$. Hence, the vector bundle $E_i$ can be deformed to a direct sum of $d_2$
   generic line bundles of degree $d_1+1$ and $r-d_2$ generic line bundles of degree $d_1$ 
   where $r=hr',\ d_2=hd'_2$. That is, we can degenerate $E_i$ to 
 $$L^1\oplus \cdots \oplus L^{d_2}\oplus \cdots \oplus L^r,$$
 where $L^1, \dots ,L^{d_2}$ have degree $d_1+1$ and $L^{d_2+1}, \dots, L^r$ have degree $d_1$.
 
 An element in the kernel of the Petri map for the generic vector bundle
  would give rise to an element of the Petri map for the degeneration. 
 Using the decomposition of $E_i$ into a direct sum, 
the Petri map restricted to the component $C_i$ splits into a direct sum of maps 
$$H^0(L^j)\otimes H^0(\omega \otimes (L^l)^*)\to H^0(L^j\otimes \omega \otimes (L^l)^*).$$
Each component of $\rho$ in this splitting should be in the kernel of the corresponding cup-product map.
 So, in each term of the decomposition, either $\rho$ is identically zero or has two non-zero terms.

If we decompose $\rho $ as in the decomposition above, the order of vanishing of $\rho$
 at a node is the minimum of 
the orders of vanishing of its $r^2$ coordinates.  The expression of a $\rho$ in the kernel of the
 Petri map will be  of the form 
$$(\dots, \sigma ^j\otimes \tau^l+ \sigma^{'j}\otimes \tau^{'l}+\cdots, \dots ).$$
Here we can assume that  $\sigma ^j, \ \sigma^{'j}$ are linearly independent. If $L^j$ has degree $d_1+1$ at most one of the two has  vanishing adding up to $d_1$ 
between  $P_i, Q_i$, the other having smaller vanishing. If $L^j$ has degree $d_1$ at most one of the two has  vanishing $d_1-1$ between  $P_i, Q_i$, the other having smaller vanishing. 

 Similarly, as $\tau ^l, \ \tau^{'l}$ can be assumed to be linearly independent.
 If $L^l$ has degree $d_1+1$ at most one of the two has  vanishing $2g-2-d_1-2=2g-d_1-4$ between  $P_i, Q_i$,
  the other having smaller vanishing.
   If $L^j$ has degree $d_1$ at most one of the two has  vanishing $2g-d_1-3$ between  $P_i, Q_i$,
    the other having smaller vanishing. 
    
    It follows that the inequalities in equation \ref{ineq} are strict for both the $\sigma $ and the $\tau $.
    Therefore the result follows from the definition of the order of vanishing for $\rho$. 
 \end{proof}

\begin{pr}\label{+Ld}
Let $C_i$ be an elliptic curve such that the restriction of the
vector bundle to $C_i$ is the direct sum of line bundles $L^j_i$ all of
degree $d_1$. Then,
$$ord_{P_{i+1}}(\rho_{i+1})\ge ord_{P_i}(\rho_i)+1.$$
The inequality is strict unless  the non-zero terms of $\rho_i$ that give the
vanishing at $P_i$ can be written as
$$(\dots, \sigma^j _0\otimes \tau^l+  \sigma^j\otimes\tau^l_0, \dots),$$
 with $\sigma^j _0,\tau^l _0$ being  the unique  sections of $L^j_i$, $\omega\otimes( L^l_i)^*$
that vanish at $P_i, Q_i$ with orders adding up to $d_1, \ 2g-2-d_1$. 

Assume also that the line bundles $L^j_i$ are either generic or of the form ${\mathcal O}(\lambda_iP_i+(d_1-\lambda_i)Q_i)$
for a fixed $\lambda_i$ that does not depend on $j$. Let $n(i)$ be the number of elliptic curves before $C_i$. If $ord_{P_i}(\rho_i)\ge 2n(i)-2$, then 
$$ord_{P_{i+1}}(\rho_{i+1})\ge 2n(i).$$

 When the inequalities for the
order of vanishing at $P_{i+1}$ are equalities, the terms of
$\rho_{i+1}$ that give the minimum vanishing at $P_{i+1}$ glue with
the terms of $\rho _i$ that give the minimum vanishing at $P_i$.
\end{pr}
\begin{proof}
Write the restriction of the vector bundle to  $C_i$  as $L^1_i\oplus \cdots \oplus L^r_i$, where $L^j_i$ are line bundles of
degree $d_1$. 

Let $\rho =\sum f_{j,l}\sigma^j\otimes \tau ^l$ be an element in the kernel of the Petri map written in
terms of a basis $\sigma^j$ of $\pi_*{\mathcal E}_i$ and a basis $\tau^l$ of $\pi_*({\omega \otimes \mathcal E^*}_i)$ such that the $t^{\alpha
_j}\sigma ^j$ is a basis of $\pi_*{\mathcal E}_{i+1}$ and $t^{\beta
_l}\tau^l$ is a basis of $\pi_*({\omega \otimes \mathcal E^*}_{i+1})$. As the map is given by cup-product, 
an element of the form $\sigma^j\otimes
\tau^l$ is not in the kernel of the Petri
map. Hence, any $\rho$ in the kernel has in its expression at least
two terms of this type and we can assume that  the $\sigma$'s appearing in the expression  are
independent.

From   Prop. \ref{vansigma} and Cor. \ref{Remark}, $ord_{P_i}(\sigma^j)\le \alpha_j\le ord_{P_{i+1}}t^{\alpha
^j}\sigma^j$, and the first inequality being an equality for at most
$r$ independent sections $\sigma ^j_0$ that vanish at
$P_i,Q_i$ with orders adding up to $d_1$. Similarly, $ord_{P_i}(\tau^l)\le \beta_l\le ord_{P_{i+1}}t^{\beta
^l}\tau^l$, with the first inequality being an equality for at most
$r$ independent sections $\tau ^l_0$ that vanish at
$P_i,Q_i$ with orders adding up to $2g-2-d_1$.

Using the decomposition of $E$ into a direct sum, 
the Petri map restricted to the component $C_i$ splits into a direct sum of maps 
$$H^0(L^j_i)\otimes H^0(\omega \otimes (L^l_i)^*)\to H^0(L^j_i\otimes \omega \otimes (L^l_i)^*).$$
Each component of $\rho$ in this splitting should be in the kernel of the corresponding cup-product map. 
So, in each term of the decomposition, either $\rho$ is identically zero or has two non-zero terms.

Note that if we decompose $\rho $ as in the decomposition above, the order of vanishing of $\rho$ at a node is the minimum of the orders of vanishing of its $r^2$ coordinates.  The expression of a $\rho$ in the kernel of the Petri map, will be  of the form 
$$(\dots , \sigma ^j\otimes \tau^l+ \sigma^{'j}\otimes \tau^{'l}+\cdots, \dots).$$
Here we can assume that  $\sigma ^j, \ \sigma^{'j}$ are linearly independent.
 Then, at most one of the two has sum of vanishing at $P_i,\ Q_i$ being $d_1$. 
  Similarly, as $\tau ^l, \ \tau^{'l}$ can be assumed to be linearly independent.
   So, at most one of the two has sum of vanishing at $P_i,\ Q_i$ of $2g-2-d_1$. 
   Therefore, for at least one of $\sigma ^j, \sigma^{'j}$,
    $ord_{P_i}(\sigma^j)\le \alpha_j-1\le ord_{P_{i+1}}t^{\alpha^j}\sigma^j-1$, and similarly for the $\tau$. 

Hence, $ord_{P_{i+1}}(\rho _{i+1}) \ge ord_{P_i}(\rho _i)+1$ and
$ord_{P_{i+1}}(\rho _{i+1}) \ge ord_{P_i}(\rho _i)+2$ except in the
case when each line bundle is special and in the splitting of the Petri map, the expression of each 
component of $\rho$ is of the form $ \sigma ^j_0\otimes \tau^l+ \sigma^j\otimes \tau^l_0$
where $\sigma ^j_0,\  \tau^l_0$ are the sections of $L^j, \omega\otimes  (L^l)^*$
that vanish with orders at $P_i,Q_i$ adding up to $d_1,\ 2g-2-d_1$ respectively.

If the line bundles are generic, there are no sections that vanish at the nodes with orders adding up to $d_1$.
So, the order of vanishing of $\rho$ increases in at least two units. 

Assume now that the line bundles are special and all the same
 and that $ord_{P_i}(\rho_i)\ge 2n(i)-2$. We want to show that 
$$ord_{P_{i+1}}(\rho_{i+1})\ge 2n(i).$$
From the discussion above, we can assume that the restriction of $\rho$ to each line bundles is of the form
$ \sigma ^j_0\otimes \tau^l+ \sigma^j\otimes \tau^l_0$,  where $\sigma ^j_0$ (resp.  $\tau^l_0$)
 are the sections of $L^j$ (resp. $\omega\otimes  (L^l)^*$),
that vanish with orders at $P_i,Q_i$ adding up to $d_1$ (resp. $2g-2-d_1$). Then, 
$\sigma ^j_0 \tau^l_0$ is a section of $\omega_i={\mathcal O}((2n(i)-2)P_i-2(g-n(i))Q_i)$.
 As $\sigma ^j_0 $ vanishes to order  adding up to $d_1$ between $P_1, Q_1$ and 
$ \tau^l_0$ vanishes to order $2g-2-d_1$  between $P_i, Q_i$,  $\sigma ^j_0\tau^l_0$ vanishes to order $2g-2$ between $P_i, Q_i$. As $P_i, Q_i\in C_i$ are generic,
 the only section of 
${\mathcal O}((2n(i)-2)P_i-2(g-n(i))Q_i)$ vanishing to order precisely $2g-2$ between $P_i, Q_i$, vanishes to order precisely $2n(i)-2$ at $P_i$.
Then,  
$$2n(i)-2\le ord_{P_i}(\rho_i)=ord_{P_i}(\sigma ^j_0)+ord_{P_i}( \tau^l)\not= 
ord_{P_i}(\sigma ^j_0)+ord_{P_i}( \tau^l_0)=2n(i)-2.$$

Hence, $ord_{P_i}(\rho_i)\geq 2n(i)-1$ and therefore, 
$$ord_{P_{i+1}}(\rho_{i+1})\geq ord_{P_i}(\rho_i)+1\geq 2n(i)$$
as needed.
\end{proof}

\vskip3mm

A similar argument in the case of a rational curve gives  the following
\vskip2mm

\begin{cor}\label{racional} If $C_i$ is a rational component, we have that $$ord_{P_{i+1}}(\rho_{i+1})\ge ord_{P_i}(\rho_i).$$
\end{cor}

\vskip 3mm

Now  the proof of the Theorem goes as follows: on every rational curve, the vanishing of a section in the kernel does not decrease (Cor. \ref{racional}) 
while on an elliptic curve, the vanishing increases in two units (Prop. \ref{+Ld}, Prop.\ref{E1}). Therefore,  on the last component of the central fiber 
 (that we can assume to be rational), we have
that $ord_{P_g}(\rho_g)\ge 2g$. As this multiplicity is at most the degree of $\omega$ which is $2g-2$, this is impossible. Hence $\rho$ cannot exist and 
the Petri map is injective.  


\end{document}